\documentclass[11pt,a4paper]{amsart}
\usepackage[all]{xy}
\usepackage{amssymb}
\usepackage{amsthm}
\usepackage{braket}
\SelectTips{cm}{}
\def\serieslogo@{} 
\def\@setcopyright{} 
\makeatother

\title[Representation Type of EI-categories]{Representation Type of EI-categories}

\author{Karsten Dietrich}
\address{Karsten Dietrich\\ Fakult\"at f\"ur Mathematik\\
Universit\"at Bielefeld\\ D-33501 Bielefeld\\ Germany.}
\email{karsten.dietrich@gmx.net}


\newtheorem{lem}{Lemma}[section]
\newtheorem{prop}[lem]{Proposition}
\newtheorem{cor}[lem]{Corollary}
\newtheorem{thm}[lem]{Theorem}

\theoremstyle{definition}
\newtheorem{exm}[lem]{Example}
\newtheorem{Rem}[lem]{Remark}
\newtheorem{defn}[lem]{Definition}

\newtheorem*{ackn}{Acknowledegement}

\renewcommand{\mod}{\operatorname{mod}\nolimits}
\newcommand{\rep}{\operatorname{rep}\nolimits}
\newcommand{\ch}{\operatorname{char}\nolimits}

\newcommand{\Fun}{\operatorname{Fun}\nolimits}
\newcommand{\Aut}{\operatorname{Aut}\nolimits}

\newcommand{\Ob}{\operatorname{Ob}\nolimits}
\newcommand{\Mor}{\operatorname{Mor}\nolimits}

\newcommand{\id}{\operatorname{id}\nolimits}

\newcommand{\Mod}{\operatorname{Mod}\nolimits}

\renewcommand{\mod}{\operatorname{mod}\nolimits}
\newcommand{\End}{\operatorname{End}\nolimits}

\renewcommand{\dim}{\operatorname{dim}\nolimits}

\newcommand{\Rep}{\operatorname{Rep}\nolimits}

\def\C{{\mathcal C}}
\def\D{{\mathcal D}}

\def\G{{\mathcal G}}

\begin{document}

\title{Representation Type of EI-Categories}

\begin{abstract}
	EI-categories are a simultaneous generalisation of finite groups and finite quivers without oriented cycles. It is therefore a natural question to ask for a characterisation of finite representation type. For special classes of EI-categories a complete characterisation is obtained using quiver techniques. For EI-categories with two objects we present a necessary criterion for finite representation type. The complexity of this classification problem is illustrated by some examples.
	\end{abstract}

\maketitle
\setcounter{tocdepth}{1}
\tableofcontents

\section{Introduction}
An EI-category $\C$ is a category in which every endomorphism is an isomorphism. For a fixed base ring $k$ the associated category algebra to a finite EI-category $\C$ is denoted by $k\C$. It has as basis the set of morphisms in $\C$ with multiplication induced by composition of morphisms. Hence, the category algebra of a finite EI-category  is a simultaneous generalisation of several important constructions in representation theory, such as the group algebra of a finite group, the path algebra of a finite quiver without oriented cycles or the incidence algebra of a finite poset.

It is therefore a natural question to ask for a classification of the representation-finite EI-categories or more precisely a classification of the finite EI-categories which lead to to a representation-finite category algebra over a field $k$. 

As this is the first serious investigation of representation type of EI-categories, we will start with a couple of basic observations. The first one will lead us to the construction of the endotrivialisation of a given EI-category which is an EI-category with only trivial endomorhpisms attached to the  given EI-category. The endotrivialisations of representation-finite EI-categories will turn out to be finite posets of finite representation type.

The analysis of examples will illustrate that, as one should expect, the representation type of EI-categories (also those not arising as group algebras) depends on the characteristic of the underlying field. A lack of methods will also become visible in the examples since neither quiver nor poset or group techniques apply for a distinction of finite and infinite representation type.

Consequently, we will restrict ourselves to a special class of EI-categories namely the EI-categories with two non-isomorphic objects. In this setting a necessary criterion for finite representation type is obtained. Roughly speaking the theorem states that for any representation-finite EI-category $\C$ there are no two non-isomorphic objects $x$ and $y$ in $\C$ such that the product $\Aut(x)\times\Aut(y)$ acts freely on $\C(x,y)$.

Finally, we will give a full classification of all representation-finite EI-categories that admit only two simple representations using quiver techniques. In particular we will calculate the Ext-quiver for the associated category algebras and apply covering theory in the sense of Bongartz and Gabriel.

\begin{ackn}
This paper is based on parts of my PhD-thesis under the supervision of Henning Krause. I am grateful to him for coming up with an interesting topic and also for his advice. 

Further I would like to thank Fei Xu whose work on representation theory of EI- categories together with personal discussions with him helped me to develop a feeling for EI-categories and their representations as well as some of the crucial ideas. 

Finally, I want to express my thanks to Peter Webb who, after reading the relevant parts of my thesis, provided several comments on my work that enlarged my understanding of the whole mathematical framework and led to the correction of one of my theorems.
\end{ackn}

\section{Basic facts and endotrivialisations}
EI-categories and their representations first appeared in work of L\"uck and tom Dieck in the 1980s in the context of algebraic K-theory. The whole theory can be embedded in the representation theory of small categories. We will restrict ourselves to finite EI-categories from the very beginning, hence not all definitions and facts are presented in full generality.

Throughout fix a field $k$.

\subsection{Basic facts about EI-categories}
\begin{defn}An \emph{EI-category} is a category $\C$ in which every endomorphism is an isomorphism. If $\mathcal{C}$ is a finite EI-category, the associated $k$-algebra  \[ k\mathcal{C} = \Set{ \sum_{f \in \Mor\mathcal{C}} \lambda_f f | \lambda_f \in k }\]
is a finitely generated unital $k$-algebra, sometimes called the associated EI-algebra. The unit element is $ \sum_{x \in \Ob\mathcal{C}} 1_x $ and obviously the elements $\Set{ 1_{x} | x \in \Ob\mathcal{C} }$ form a set of pairwise mutually orthogonal idempotents in $k\mathcal{C}$. These idempotents are in general not primitive.
\end{defn}
If we speak of EI-categories we will always refer to finite EI-categories. A representation $V$ of an EI-category $\C$ over $k$ is a covariant functor $V: \C \to \mod k$. Morphisms of representations are morphisms of functors. Hence, the representations of EI-categories form an abelian category.

The following elementary observation relates the concepts of representations of $\C$ and modules over $k\C$.
\begin{prop}[Mitchell,\cite{Mitchell}]
Let $\C$ be a category with finitely many objects and $k$ a field. Then the categories $\Rep_k\C$ and $\Mod k\C$ are equivalent. This equivalence restricts to an equivalence $\mod k\C \to \rep_{k}\C = \Fun(\C,\mod k)$ on the finite-dimensional level.
\end{prop}
\begin{exm}
The motivating examples for our work on EI-categories are the following ones.	
\begin{itemize}
\item[(1)] Let $G$ be a finite group and let $\G$ be the category with one object $x$ and $\text{End}(x)= G$. Then $\G$ is an EI-category and $k\G = kG$.
\item[(2)] Let $Q$ be a finite quiver without oriented cycles and $\underline{Q}$ its path category. Then $\underline{Q}$ is an EI-category with $k\underline{Q} = kQ$.
\end{itemize}
Another important example of EI-categories is given by partially ordered sets. The associated category has only trivial endomorphisms and is therefore EI. These particular EI-categories will play a prominent role later on.

The second branch of mathematics where EI-categories arise is algebraic topology. Fusion categories and transporter categories, recently studied by Broto, Levi and Oliver \cite{BLO}, and orbit categories, playing an important role in the theory of finite G-spaces, are all EI-categories.
\end{exm}
The set of isomorphism classes of objects of an EI-category $\C$ carries the natural structure of a finite partially ordered set induced by the relation $$ x \leq y \Leftrightarrow \C(x,y) \neq \emptyset. $$
This poset structure will be used in the characterisation of representation-finite endotrivial EI-categories.
\subsection{The endotrivialisations of EI-categories of finite type}
To an arbitrary EI-category $\C$ we will associate another category $\widehat{\C}$ with only identity endomorphisms, which reflects the global structure of $\C$. This category $\widehat{\C}$ will play an important role in the analysis of the representation type of EI-categories.
\begin{defn}
Let $f$ and $g$ be two morphisms in a finite, skeletal EI-category $\C$. Then we define a relation $\sim$ on the set of morphisms of $\C$ as follows.
$$f \sim g :\Leftrightarrow f = f''h_1 f'  \text{ and }  g = f'' h_2 f'   \text{ for some }  f'', f' \in \Mor\mathcal{C}  \text{ and endomorphisms }  h_1, h_2 $$
This is clearly a reflexive and symmetric relation. We will consider the transitive hull of this relation and denote it again by $\sim$. This relation is also compatible with the composition of morphisms in $\mathcal{C}$. Therefore, we get a new category $ \widehat{\mathcal{C}}:= \mathcal{C}/\sim $ which is by construction an endotrivial category (in particular EI).
\end{defn}
 Roughly speaking, $\widehat{\mathcal{C}}$ is constructed from $\mathcal{C}$ by making all endomorphisms trivial and identifying all morphisms $x \to y$ in the same $(\Aut(x)\times \Aut(y))$-orbit. By construction it has the following important universal property. Suppose that $\C$ is an EI-category and $F:\C \to \D$ any functor to an endotrivial category $\D$. Then this functor $F$ factors via a unique functor through the quotient functor $G: \C \to \widehat{\C}$, i.e.\ the following diagram is commutative. $$ \begin{xy}\xymatrix{ \C \ar[r]^G \ar[d]_F & \widehat{\C} \ar@{-->}[dl]^{\exists!} \\ \D & }\end{xy}$$
 \begin{exm}
\begin{enumerate}
\item[(1)] If $G$ is a finite group and $\mathcal{C} = \G$ the associated EI-category, then $\widehat{\G}$ consists of one object $x$ and the only morphism is the identity $1_x$.
\item[(2)] If $Q$ is a finite quiver without oriented cycles and $\mathcal{C} = \underline{Q}$ the path category, then $\widehat{\mathcal{C}} = \mathcal{C}$.
\item[(3)] If $\mathcal{C}$ is the EI-category associated to a finite poset $(X,\leq)$ we get $\widehat{\mathcal{C}} = \mathcal{C}$.
\end{enumerate}
\end{exm}
\begin{Rem}\begin{enumerate}
\item[(1)]
An EI-category $\C$ is not uniquely determined by the category $\widehat{\C}$ together with its automorphism groups. To recover the entire structure of $\C$ one needs to know the composition of morphisms which is the same as the whole structure of $\C$. Nevertheless, the category $\widehat{\C}$ is of great importance for us. As an example consider the EI-category \[\mathcal{C}: \begin{xy}\xymatrix{ x\ar@(ul,dl)_{f} \ar@/^0.8cm/[rr]^{i_1}\ar@/^0.3cm/[rr]^{i_2}\ar@/_0.3cm/[rr]^{i_3}\ar@/_0.8cm/[rr]^{i_4} &  &y   ,} \end{xy} \]
satisfying the relations $f^4 = 1_x$ and $i_1 f = i_2$, $i_2 f = i_3$, $i_3 f = i_4$, $i_4 f = i_1$ and the EI-category $$\C': \begin{xy}\xymatrix{ a \ar@(ul,dl)_g \ar[r]^{h} & b ,}\end{xy}$$ with the relations $g^4 = 1_a$ and $gh = h$. Then both $\widehat{\C}$ and $\widehat{\C'}$ are the path category of $A_2$ and $\C$ and $\C'$ have the same automorphism groups. However, they are not equivalent and the associated category algebras have completely different representation-theoretic properties. We will later see that $\C$ is representation-infinite and $\C'$ has finite representation type over any algebraically closed field $k$.
\item[(2)] By construction of $\widehat{\C}$ we have the quotient functor $G:\C \to \widehat{\C}$ which is the identity on objects and surjective on morphisms. This functor induces a fully faithful embedding of $\mod k\widehat{\C}$ into $\mod k\C$. The existence of this embedding implies that if $k\widehat{\C}$ is of infinite representation type then $k\C$ is of infinite type as well. It is therefore natural to ask for a classification of representation-finite EI-categories with only trivial endomorphisms as a starting point.
\end{enumerate}
\end{Rem}
\begin{thm}\label{endotrivial}
Let $\mathcal{C}$ be an EI-category with only trivial endomorphisms. Then $\mathcal{C}$ is of finite representation type if and only if $k\mathcal{C}$ is Morita-equivalent to an incidence algebra of finite representation type.

\end{thm}

\begin{proof}
Let $\mathcal{C}$  be endotrivial and of finite representation type. We have to show that $k\mathcal{C}$ is Morita-equivalent to an incidence algebra of finite type. The other direction in the theorem is trivial. Again, we may (up to Morita-equivalence of $k\mathcal{C}$) assume that $\mathcal{C}$ is skeletal.\\
Since $\mathcal{C}$ is representation-finite, there are no objects $x,y \in \mathcal{C}$ with two distinct morphisms $f,g: x \to y$. Otherwise one gets a fully faithful embedding of the category of representations of the $2$-Kronecker into $\mod k\mathcal{C}$. Since $\mathcal{C}$ is skeletal its set of objects carries a natural structure of a finite poset defined by 
\[ x \leq y :\Leftrightarrow \exists f \in \mathcal{C}(x,y). \]
Using, that $\C(x,y) \leq 1$ for all $x,y \in \Ob \C$, we get that $k\mathcal{C}$ is the incidence algebra associated to the finite poset $(\Ob\mathcal{C},\leq)$ and the claim follows.
\end{proof}
\begin{cor}
Let $\C$ be a finite and skeletal EI-category and $k$ a field such that $k\C$ is representation-finite. Then $\widehat{\C}$ is a finite poset of finite representation type.
\end{cor}
\begin{Rem}
Suppose $\C$ is an endotrivial EI-category. Then its set of objects naturally carries the structure of a finite partially ordered set, but in general it will not be the case that $\C$ is the category associated to this poset. For instance, take $\C$ to be the category 
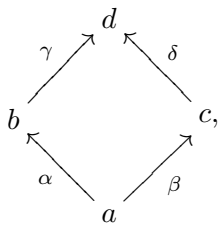
\[ \begin{xy}\xymatrix{ & d &  \\ b\ar[ur]^{\gamma} & & c \ar[ul]_{\delta}, \\ & a \ar[ul]^{\alpha}\ar[ur]_{\beta} & }\end{xy} \]  without any relations. Then the category associated to $(\Ob\C,\leq)$ is $\C$ modulo the relation $\gamma\alpha = \delta\beta$.
\end{Rem}
\section{EI-categories with two objects: first observations}
The easiest class of EI-categories for which the representation type has not yet been investigated is the class of EI-categories with two non-isomorphic objects. Clearly, a representation-finite EI-category $\C$ with two objects has to satisfy $\widehat{\C} = A_2$ and both group algebras attached to the two objects have to be of finite representation type.

An arbitrary EI-category algebra $k\C$ is representation-infinite if there exists at least one full subcategory $\D \subseteq \C$ such that the algebra $k\D$ is representation-infinite. Therefore, the treatment of EI-categories with two objects yields necessary criteria for an EI-category algebra to be representation-finite. 

\subsection{The easiest example}
We will now discuss an example of an EI-category to illustrate the difficulties one faces when dealing with these categories.

We consider an EI-category $\mathcal{C}$ with two objects $x$ and $y$ such that $\End(x) = \langle f \rangle \cong \mathbb{Z}_2$ and $\End(y) = \langle g \rangle \cong \mathbb{Z}_2$. Under these assumptions there are 5 different EI-categories that can appear, namely:
\begin{itemize}
\item[(1)] $\mathcal{C}(x,y) = \Set{ i }$ with $i\circ f = i$ and $g \circ i = i$;
\item[(2)] $\mathcal{C}(x,y) = \Set{i_1, i_2 }$ with $f$ acting trivially on $\Set{i_1, i_2 }$ and $g$ permuting the two morphisms;
\item[(3)] $\mathcal{C}(x,y) = \Set{i_1, i_2 }$ with $f$ permuting $i_1$ and $i_2$ and $g$ acting trivially;
\item[(4)] $\mathcal{C}(x,y) = \Set{i_1, i_2 }$ with $f$ and $g$ permuting $i_1$ and $i_2$;
\item[(5)] $\mathcal{C}(x,y) = \Set{i_1,i_2,i_3,i_4}$ with $i_1 \circ f = i_2$, $i_3 \circ f = i_4$, $g \circ i_1 = i_3$ and $g \circ i_2 = i_4$.
\end{itemize}
The second case has briefly been studied by Xu in \cite{Fei1}, where he claims that this category is of infinite representation type in characteristic $2$, but the representations he constructed turn out to be decomposable and (as we will see later) and the category is of finite representation type. It will turn out that the representation type of an EI-category with two objects is mainly governed by the group action of the endomorphism groups on the set of morphisms between the two non-isomorphic objects. 

The representation type of the 5 cases from above is computed via the calculation of the Ext-quiver and then the use of quiver techniques. This leads to the following results.
\begin{itemize}
\item[(1)]
The computation of the Ext-quiver of $k\mathcal{C}$ in characteristic $2$ yields the quiver
$$ Q: \begin{xy}\xymatrix{ \circ \ar@(ul,dl)_{\alpha} \ar[r]^{\beta} & \circ \ar@(ur,dr)^{\gamma} } \end{xy}.$$
We fix this quiver for the remainder of this subsection.

An easy calculation shows that $k\mathcal{C}$ is isomorphic to $kQ/I$, where $I$ is the admissible ideal generated by the zero relations $ 0 = \alpha^2 = \gamma^2 = \beta \alpha = \gamma \beta $. Therefore, $k\mathcal{C}$ is a string algebra without any bands and those algebras are known to be of finite representation type. Alternatively one can consult the list at the end of \cite{Bongartz-Gabriel} and see that this algebra is of finite representation type and also find its Auslander-Reiten quiver. 

If char$(k) \neq 2$, then the algebra $k\mathcal{C}$ is isomorphic to the path algebra $kQ$ where $Q$ is the quiver
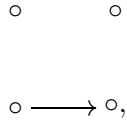
\[ \begin{xy}\xymatrix{ \circ & \circ \\ \circ \ar[r] & \circ, }\end{xy}\]
which is obviously of finite representation type.
The other cases are discussed analogously and we deduce the following results (with respect to the numbering from the beginning of this subsection). 
\item[(2)] $k\mathcal{C} \simeq kQ/\langle \alpha^2, \gamma^2, \beta\alpha \rangle $ if char$(k) = 2$ and this algebra is of finite representation type, again as  a string algebra without any bands. If the characteristic is different from $2$, then $k\mathcal{C}$ is the path algebra of the representation-finite quiver
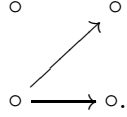
\[ \begin{xy}\xymatrix{ \circ & \circ \\ \circ \ar[ur] \ar[r] & \circ. }\end{xy}\]
\item[(3)] This situation is dual to the one above. In char$(k) = 2$ we have $k\mathcal{C} \simeq kQ/\langle \alpha^2, \gamma^2, \gamma \beta \rangle$, which again is a string algebra without bands, and in char$(k) \neq 2$ the algebra $k\mathcal{C}$ is hereditary of finite representation type as the path algebra of the following quiver
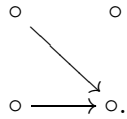
\[ \begin{xy}\xymatrix{ \circ \ar[dr]& \circ \\ \circ  \ar[r] & \circ. }\end{xy}\]
\item[(4)] If char$(k) = 2$, then $k\mathcal{C} \simeq kQ / \langle \alpha^2, \gamma^2, \beta \alpha - \gamma \beta \rangle$. This algebra is representation-finite, which can be seen by knitting the AR-quiver using covering theory as we explained before (or see for example \cite{Gabriel}). If the characteristic is different from $2$, then the algebra $k\mathcal{C}$ is isomorphic to the path algebra of the quiver
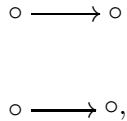
\[\begin{xy}\xymatrix{ \circ \ar[r] & \circ \\ \circ  \ar[r] & \circ, }\end{xy}\]
which is of finite representation type.
\item[(5)] If char$(k) = 2$, then we get that $k\mathcal{C} \simeq kQ/\langle \alpha^2, \gamma^2 \rangle$ which is a string algebra with infinitely many bands and therefore it is of infinite representation type. In any characteristic different from $2$ the algebra $k\mathcal{C}$ is hereditary (as we know from work of Xu) and in this particular case isomorphic to the path algebra of the following quiver
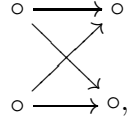
\[\begin{xy}\xymatrix{ \circ \ar[r]\ar[dr] & \circ \\ \circ \ar[ur] \ar[r] & \circ, }\end{xy}\]
which is a quiver with underlying Eucledian graph $\tilde{A}_3$.
\end{itemize}
\subsection{The characteristic plays a role}
In this subsection we present an example, which shows that the representation type of an EI-category algebra $k\C$, which does not come from a group, depends on the characteristic of the field $k$. The reason for this is that the same is true for group algebras. The easiest example where different characteristics of $k$ yield different representation types is the following one.
\begin{exm}
Consider the EI-category
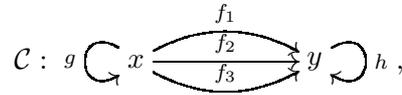
\[\mathcal{C}: \begin{xy}\xymatrix{ x\ar@(ul,dl)_{g} \ar@/^0.4cm/[rr]^{f_1}\ar[rr]^{f_2}\ar@/_0.4cm/[rr]^{f_3} &  &y \ar@(ur,dr)^{h} } \end{xy}, \]
with the relations $g^2 = 1_x, \;$ $f_i g = f_i,\;$ for $ i=1, 2, 3, \;$  $h^3 = 1_y$ and $hf_1 = f_2, \;$  $hf_2 = f_3,\;$  $hf_3 = f_1$. \\
We will now show that the associated $k$-algebra $k\mathcal{C}$ is of finite representation type if and only if char$(k) \neq 2$. First of all, we suppose that char$(k) \neq 2,3$. Then $k\mathcal{C}$ is hereditary (and basic) and we compute the Ext-quiver to be 
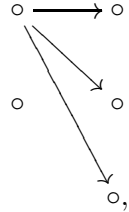
\[\begin{xy}\xymatrix{ \circ \ar[r] \ar[dr] \ar[ddr] & \circ \\ \circ & \circ \\  & \circ, }\end{xy}\]
which is of finite representation type as the union of quivers with underlying graphs $A_1$ and $D_4$. 

Now we assume that the characteristic of $k$ is $3$. Then the radical of $k\mathcal{C}$ is $\langle 1_y - h, f_1, f_2, f_3 \rangle$ and rad$^2 k\mathcal{C} = \langle (1_y - h)^2, f_1 - f_2, f_2 - f_3 \rangle$ where $\langle$ --- $\rangle$ denotes the $k$-span. A complete list of primitive, orthogonal idempotents is given by $1_y, \frac{1}{2}(1_x + g), \frac{1}{2}(1_x -g)$. Therefore, $k\mathcal{C}$ is isomorphic to the path algebra of the quiver 
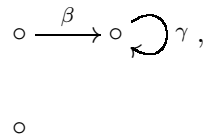
\[ \begin{xy}\xymatrix{ \circ \ar[r]^\beta & \circ \ar@(ur,dr)^{\gamma} \\ \circ & \;  }\end{xy},\]
bound by the relation $\gamma^3 = 0$. This bound path algebra is of finite representation type since it is the union of $A_1$ with a bound path algebra whose module category can be embedded into the module category of $k(\begin{xy}\xymatrix{ \circ \ar@<5pt>[r]^{\nu} & \circ \ar@<5pt>[l]^{\delta} \ar@(ur,dr)^{\rho}}\end{xy}) / \langle \rho^3,\delta\rho,\delta\nu,\delta\rho\nu \rangle$. The latter is known to be of finite type by work of Bautista,  Gabriel, Roiter and Salmeron \cite{Bautistaetal} and work of Bongartz and Gabriel \cite{Bongartz-Gabriel}. 

The last case is the one where $\ch(k) = 2$. Here the radical of $k\mathcal{C}$ is $\langle 1_x + g,f_1,f_2,f_3 \rangle$ and its square is zero. A complete list of primitive, orthogonal idempotents is given by $1_x, \frac{1}{3}(1_y + h + h^2), \frac{1}{3}(1_y + \varepsilon h + \varepsilon^2 h^2), \frac{1}{3}(1_y + \varepsilon^2 h + \varepsilon h^2)$, where $\varepsilon$ denotes a primitive third root of unity in $k$. Therefore, we deduce for the Ext-quiver of $k\mathcal{C}$ the quiver 
\[ \begin{xy}\xymatrix{& & \circ \\ \Gamma: &\circ \ar@(ul,dl)_{\alpha} \ar[ur]^{\beta_1} \ar[r]^{\beta_2} \ar[dr]_{\beta_3} & \circ ,\\ & & \circ }\end{xy}\]
and $k\mathcal{C}$ is isomorphic to $k\Gamma / \langle \alpha^2, \beta_1\alpha, \beta_2\alpha, \beta_3\alpha \rangle$. This bound quiver is of infinite representation type since its universal cover contains the $4$-subspace quiver without relations. It is also not difficult to write down a $1$-parameter family of indecomposables of dimension vector $(3,1,1,1)$.
\end{exm}
The discussion of the first case in the example can be generalised to the following Proposition, which provides us with a situation where we can prove that the EI-category algebras in question always have finite representation type.
\begin{prop}
Let $\mathcal{C}$ be a finite, skeletal EI-category with 2 objects $x$ and $y$ such that $\mathcal{C}(x,y) = \Set{f} $ and the two groups $\End(x)$ and $\End(y)$ are abelian. Let $k$ be an algebraically closed field whose characteristic neither divides the order of $\End(x)$ nor the order of $\End(y)$. Then $k\mathcal{C}$ is of finite representation type.
\end{prop}
\begin{proof}
Let $n:= |\End(X)|$ and $m:= |\End(Y)|$. 

First of all we note that the algebra $k\mathcal{C}$ is hereditary (see Theorem 4.2.4 in \cite{Fei1}) and basic since the two groups are abelian. Therefore, it is isomorphic to the path algebra of its Ext-quiver. 

Since the groups $\End(x)$ and $\End(y)$ are assumed to be abelian of order not divisible by the characteristic of our field, it is known that $k\mathcal{C}$ has exactly $m+n$ isomorphism classes of simple modules (see \cite{Serre} and \cite{Lueck}) and all the simples do not have self-extensions. This implies, that the Ext-quiver $\Gamma(k\mathcal{C})$ has $m+n$ vertices. Furthermore, we know that the $k$-dimension of $k\mathcal{C}$ is $m+n+1$. This yields that we have exactly one arrow in $\Gamma(k\mathcal{C})$ and the claim follows.
\end{proof}
\section{Free action implies infinite type}
In this section we prove the fact that free action of the automorphism groups of an EI-category $\C$ with two objects $x$ and $y$ on $\C(x,y)$ implies infinite representation type in any characteristic. This is the most general result we will achieve in our treatment of EI-categories with two objects. The proof is carried out by considering various cases. The most interesting cases will be presented as Lemmata starting with the following one.
\begin{lem}
Let $\mathcal{C}$ be an EI-category with two non-isomorphic objects $x, y$ and abelian endomorphism groups $\End(x)$ and $\End(y)$ of order $\geq 2$ such that the group action of $\End(x) \times \End(y)$ on $\mathcal{C}(x,y)$ is free and transitive. Let $k$ be an algebraically closed field which characteristic neither divides the order of $\End(x)$ nor the order of $\End(y)$. Then $k\mathcal{C}$ is of infinite representation type.
\end{lem}
\begin{proof}
Analogous to the proof in the last section, we have that $k\mathcal{C}$ is isomorphic to the path algebra of its Ext-quiver which has $m+n$ vertices. Since $\dim_k k\mathcal{C} = m + n + m\cdot n$ there are $m\cdot n$ arrows in $\Gamma(k\mathcal{C})$. Therefore, the underlying graph of the Ext-quiver is not Dynkin.
\end{proof}

If both group algebras are not semisimple, we can prove the assertion without constructing representations or computing the Ext-quiver.
\begin{lem}
Let $\mathcal{C}$ be an EI-category with two objects $x$ and $y$ and let $k$ be an algebraically closed field of positive characteristic $p$ dividing both $|\End(x)|$ and $|\End(y)|$. Further, we assume that $\End(x) \times \End(y)$ acts freely on $\mathcal{C}(X,Y)$. Then $k\mathcal{C}$ is of infinite representation type.
\end{lem}
\begin{proof}
For simplicity we write $G:= \End(x)$ and $H:= \End(y)$. We will prove the theorem by constructing a fully faithful embedding $F:\mod k(G \times H) \to \mod k\mathcal{C}$. The construction of this functor is rather obvious. For $M \in \mod k(G\times H)$ let $F(M)(x) = M = F(M)(y)$ together with the natural actions of $G$ and $H$ on M given by $G \times \id_Y$ and $\id_X \times H$, respectively. Furthermore, we let $F(M)(\mathcal{C}(x,y)) = G \times H$. This is indeed a representation of $k\mathcal{C}$. Now let $\mu: M \to M'$ be a morphism in $\mod k(G\times H)$. We put $F(\mu) = (\mu,\mu)$ which gives a morphism $F(M) \to F(M')$ of $k\mathcal{C}$ modules. Finally this functor is fully faithful and $k$-linear by construction and it is known that $k(G\times H)$ (in this particular framework) is of infinite representation type.
\end{proof}
\begin{Rem}
If $\C$ is a skeletal EI-category with two objects $x,y$ such that $\Aut(x)\times\Aut(y)$ acts freely on $\C(x,y)$, then we can also localise the category $\C$ with respect to the set of morphisms $S:=\C(x,y)$ in the sense of Gabriel and Zisman \cite{GabrielZisman}. This gives a new category $\C[S]$ and the category of representations of $\C[S]$ is equivalent to the subcategory of representations of $\C$ consisting of all representations $V$ for which $V(f)$ is invertible for all $f$ in $\C(x,y)$. Then, it is easy to see that this category contains $\mod(\Aut(x)\times\Aut(y))$ as a full subcategory as we have seen in the previous proof.

One should note that the localisation of an EI-category is not again an EI-category in general, but if we assume that the EI-category has no parallel morphisms (i.e. $|\C(x,y)|\leq 1$ for all $x,y \in \Ob\C$), then every localisation is again EI.

\end{Rem}
For cyclic groups the computation of the Ext-quiver is rather easy, which gives the next lemma.
\begin{lem}
Let $\mathcal{C}$ be an EI-category with two non-isomorphic objects $x$ and $y$ such that the action of $\End(x) \times \End(y)$ on $\mathcal{C}(x,y)$ is free and $\End(x)$ and $\End(y)$ are cyclic of order $\geq 2$. Then $k\mathcal{C}$ is of infinite representation type for any algebraically closed field $k$.
\end{lem}
\begin{proof}
In the case of cyclic endomorphism groups one can easily compute the Ext-quiver of $k\mathcal{C}$ and derive the relations on it such that $k\mathcal{C}$ is isomorphic to this bound path algebra. One has to distinguish between the cases where the orders of both groups are divided by the characteristic, only one of them or none. In all the three cases one ends up with infinite representation type. We are not going to present the details here, the computations are exactly the same as in the examples we discussed above.
\end{proof}
Up to now we have seen several special cases of EI-categories with two objects and two non-trivial endomorphism groups with free action which are representation-infinite. We are now in the position to prove the general statement
\begin{thm}\label{free_action}
Let $\mathcal{C}$ be a skeletal EI-category with two objects $x$ and $y$ such that the groups $\End(x)$ and $\End(y)$ are non-trivial and their product $\End(x) \times \End(y)$ acts freely on $\mathcal{C}(x,y)$. Then $k\mathcal{C}$ is of infinite representation type for any algebraically closed field $k$. 
\end{thm}
\begin{proof}
The claim has already been proven for the case where both group algebras are non-semisimple, for the case where both groups are cyclic and for the case of two semisimple abelian group algebras. To prove the theorem we still have to distinguish between different cases. For simplicity denote $G:= \End(X)$ and $H:=\End(Y)$.
\begin{itemize} 
\item[(a)] Suppose that $kG$ and $kH$ are semisimple. We are going to construct an infinite family $(V_{\lambda})_{\lambda \in k^{\star}}$ of pairwise non-isomorphic indecomposable representations of $\mathcal{C}$. Let $\lambda \in k^{\star}$ and define $V_\lambda(X) = M$, $V_\lambda(Y) = N$ where $M$ is a $kG$-module, $N$ a $kH$-module, both having dimension at least $2$. Furthermore we have to choose one linear map $M \to N$ which we want, for some fixed basis, to be given by the matrix $$A_{\lambda}:= \left(\begin{array}{ccccc} 1 & 0 & 0 &  0 &\dots \\ \lambda & 1 & 0 & 0 & \dots  \\ 0 & 0 & 0 & 0 & \cdots \\ \vdots & \vdots & \vdots & \vdots & \cdots \end{array}\right),$$ where all the dots stand for zeros. For the choice of $M$ and $N$ we again have to distinguish different cases.
\begin{enumerate}
\item[(i)] Suppose that neither $G$ nor $H$ is abelian. Then we can choose $M$ to be a simple $kG$-module and $N$ to be a simple $kH$-module, both of dimension $\geq 2$. Then, since $M$ and $N$ are indecomposable, it is clear that every $V_{\lambda}$ is indecomposable. We will now show that for $\lambda \neq \mu$ we have $V_{\lambda}\not\cong V_{\mu}$. To see that, suppose that $(\phi,\psi): V_{\lambda}\to V_{\mu}$ is an isomorphism of $k\C$-modules. This implies that $\phi \in \End_{kG}(M) \cong k$ and $\psi \in \End_{kH}(N) \cong k$, which means $\phi = \alpha \cdot 1_{M}$ and $\psi = \beta \cdot 1_N$. In addition, $\phi$ and $\psi$ have to be compatible with the action of the matrix $A_{\lambda}$ (defined above). This gives the equations $\alpha = \beta$ and $\alpha\lambda = \beta\mu$ and hence $\alpha = \beta = 0$ which contradicts the assumption that $(\phi,\psi)$ is an isomorphism.
\item[(ii)] Suppose that $H$ is non-abelian and $G$ is abelian (the other way around is dual). Choose $N$ as above  and put $M = k^2$ with $G$-action given by the matrix  $ \left(\begin{smallmatrix} a & 0 \\ 0 & b \end{smallmatrix}\right)$ where $a \neq b$ and both are non-zero. In other words we want $M$ to be the direct sum of two non-isomorphic one-dimensional simple $kG$-modules. In this case we have that $\End_{kG}(M) = \left(\begin{smallmatrix} k & 0 \\ 0 & k \end{smallmatrix} \right)$ and we deduce that $V_\lambda$ is indecomposable. As above we see that $V_\lambda \not\cong V_\mu$ for $\lambda \neq \mu$.
\end{enumerate}
\item[(b)] The last case that has to be treated (again by subdivision into different cases) is the case where one of the group algebras $kG$ and $kH$ is semisimple while the other is not and not both of them are abelian. We will deal with the case where $kG$ is not semisimple and $kH$ is semisimple, the other case can be proven analogously.

We may assume, that the Sylow $p$-subgroup $D$ of $G$ is cyclic ($p = \ch(k)$), since otherwise $kG$ and hence $k\C$ is of infinite representation type and we have nothing to prove. By standard results from representation theory of finite groups we can then choose an indecomposable $kG$-module $M$ such that its restriction $M\downarrow_{D}$ has a $p$-dimensional direct summand on which $D$ acts by the matrix $$S = \left(\begin{array}{cccccc} 1 & 1 &  &  &  &   \\ & 1 & . & &  &   \\  &  &  . & . &  &  \\  &  &  & . & .  &  \\ & & &  &  . & 1 \\  &  &  &  &  & 1 \end{array}\right).$$ If now $H$ is not abelian we choose, as above, a simple $kH$-module $N$ with $\dim N \geq 2$ and for $\lambda \in k^{\star}$ we denote by $A_{\lambda}$ the same matrix as above. Then $ \begin{xy}\xymatrix{ M \ar[r]^{A_{\lambda}} & N }\end{xy}$ is an indecomposable representation of $\C$. We should now show that $V_{\lambda}\not\cong V_{\mu}$ for $\lambda \neq \mu$. Suppose that $((b_{i,j}),(c_{i,j}))$ is an isomorhpism of representations $V_{\lambda}\to V_{\mu}$. Then the matrix $(b_{i,j})$ has to commute with the $G$-action on $M$, in particular with a matrix of the shape $\left(\begin{smallmatrix} S & 0 \\ 0 & \star \end{smallmatrix}\right)$, where $\star$ is any matrix. This gives the conditions $b_{2,1} = 0 $ and $b_{1,1} = b_{2,2}$. An endomorphism of $N$ as a $kH$-module is just a scalar multiple of the identity, i.e. $(c_{i,j}) = c \cdot 1_N$. Finally, the following diagram has to commute.
$$\begin{xy}\xymatrix{ M \ar[d]_{(b_{i,j})} \ar[r]^{A_{\lambda}} & N \ar[d]^{c} \\ M \ar[r]^{A_{\mu}} & N }\end{xy}$$
This yields the conditions $b_{1,2} = 0$, $c = b_{1,1} = b_{2,2}$ and $\lambda\cdot c = \mu \cdot c$, which give that $c = 0$ and hence $V_{\lambda} \not\cong V_{\mu}$. If the group $H$ is abelian we replace the module $N$ from above by a $2$-dimensional $kH$-module which is the direct sum of two non-isomorphic one-dimensional simple $kH$-modules and get the claim by the same computations as we have just done.

To finish the proof we should consider the case where $kG$ is semisimple and $kH$ is not and not both are abelian. In this case the argument is the same as in the case we have treated above, only the computations are a little bit different. \qedhere
\end{itemize}
\end{proof}


\section{EI-category algebras with two simple modules}
As we have seen, the classification of EI-categories of finite representation type gets very complicated, even with the assumption that the category has only two objects. The distinguishing mark for finite or infinite representation type seems to be the nature of the group actions of the automorphism groups on the morphism sets between distinct objects. In this section we will give a classification of all representation-finite EI-category algebras with only two simple modules. This work is motivated by work of Bongartz and Gabriel \cite{Bongartz-Gabriel} who classified all representation-finite $k$-categories with two simples and radical of codimension $2$. We will compute the 
Ext-quiver of a given EI-category algebra with two objects and then use the list of Bongartz and Gabriel. For details on this list we refer to \cite[page 242]{Bautistaetal}.

Until the end of the article $k$ denotes an algebraically closed field of positive characteristic $p$.

Since we want to decide how many simples an EI-category algebra $k\C$ has and the simples of $k\C$ are given by the simple modules over the group algebras $k\Aut(x)$ for $x\in\Ob\C$, the following well-known result from representation theory of finite groups is useful.
\begin{lem}[see for example \cite{Alperin}]
 Let $G$ be a finite group. Then the number of simple $kG$-modules equals the number of conjugacy classes of elements in $G$ whose order is not divisible by $p$.
\end{lem}
The elements of a finite group $G$ of order not divisible by some prime $p$ are called \emph{$p$-regular}, the remaining ones are called \emph{$p$-singular}. Every element of $G$ can be written as a product of a $p$-singular and a $p$-regular element. Using this fact, we observe the following easy but important statement.
\begin{cor}
Let $G$ be a finite group such that the group algebra $kG$ has only one simple module. Then $G$ is either the trivial group or a $p$-group for $p = \ch(k)$.
\end{cor}
\begin{proof}
From the lemma we know that $G$ has only one $p$-regular element, namely the unit $1_G$. If $p$ does not divide the order of $G$, then if follows that $G = \Set{1_G}$. Suppose that $p\mid |G|$ and let $x \in G$ be any element. Then we write $x = z y$ as a product of a $p$-regular element $z$ and a $p$-singular element $y$. By assumption we have $z = 1_G$ and therefore $x$ is $p$-singular. Hence, every element has order divisible by $p$ and $G$ is a $p$-group.
\end{proof}
Now we can give a full classification of all representation-finite EI-categories with two simples. The following list is particularly interesting in characteristic $2$ and $3$.
\begin{thm}
Let $\C$ be a skeletal EI-category and $k$ an algebraically closed field such that $k\C$ has two simple modules. Then $k\C$ is of finite representation type if and only if it satisfies one of the following conditions.
\begin{itemize}
\item[$(1)$] $\C$ has one object $x$, the group $\Aut(x) = \Mor\C$ has two conjugacy classes of $p$-regular elements and the Sylow $p$-subgroup of $G$ is cyclic.
\item[$(2)$] $\C$ has two non-isomorphic objects $x$ and $y$, $\Aut(x)$ and $\Aut(y)$ are cyclic $p$-groups or trivial groups, the action of $\Aut(x)\times\Aut(y)$ on $\C(x,y)$ has at most one orbit and $\C$ is a quotient category or its dual is a quotient category of an EI-category satisfying one of the following properties.
\begin{enumerate}
\item[(a)] $|\C(x,y)|\leq 1$.
\item[(b)] $|\Aut(x)| \cdot |\Aut(y)| \leq 3$.
\item[(c)] $\ch k =2$, $|\Aut(x)| = 2$, $|\Aut(y)| = 2^s$, $s\geq 0$ and $|\C(x,y)| = 2$ with $\Aut(x)$ acting freely and $\Aut(y)$ acting trivially or the generator of $\Aut(y)$ permuting the two elements of $\C(x,y)$.
\item[(d)] $\ch k = 3$, $|\Aut(x)| = 3 =|\Aut(y)|$, $|\C(x,y)| = 3$ and both $\Aut(x)$ and $\Aut(y)$ permute the three morphisms in $\C(x,y)$. 
\end{enumerate}
\end{itemize}
\end{thm}
\begin{proof}
$\C$ has at most two objects since every object gives at least one simple $k\C$-module. If $\C$ has only one object, then $k\C$ is a group algebra and, if it is representation-finite with only two simples, it has to satisfy condition (1).

Suppose that $\C$ has two objects $x$ and $y$. The assumption that $k\C$ has two simples implies that both $k\Aut(x)$ and $k\Aut(y)$ have one simple module. Hence, they are either trivial or $p$-groups. In case of two trivial automorphism groups, $\C$ is the path category of the Dynkin quiver $A_2$ which is representation-finite. If one of $\Aut(x)$ and $\Aut(y)$ is a $p$-group it has to be representation-finite which means that it is a cyclic $p$-group. Assume that $\Aut(x)$ is a cyclic $p$-group and $\Aut(y)$ is trivial. Then the computation of the Ext-quiver of $k\C$ yields, that $k\C$ is isomorphic to the following path algebra with relations or its dual:
$$\begin{xy}\xymatrix{\circ\ar@(dl,ul)^{\alpha}\ar[r]^{\beta} & \circ }\end{xy}, \;\;\;\;\; \alpha^m = 0 = \beta\alpha^n, \; m = p^r, \; n \mid m. $$
According to the Bongartz-Gabriel list, this algebra is representation-finite only for the following values of $m$ and $n$
\begin{itemize}
\item $m = 2$, $n=1,2$;
\item $m = 3$, $n=1,3$;
\item $m = 4$, $n=1,2$;
\item $m \geq 5$, $n = 1$.
\end{itemize}
Any of these cases fulfils one of the conditions from (2).\\
Analogously, we assume that both $\Aut(x)$ and $\Aut(y)$ are cyclic $p$-groups. In this case $k\C$ is isomorphic to one of the following path algebras with relations or its dual:
\begin{itemize}
\item[(i)]
$$\begin{xy}\xymatrix{\circ\ar@(dl,ul)^{\alpha}\ar[r]^{\beta} & \circ \ar@(dr,ur)_{\gamma} }\end{xy}, \;\;\;\;\; \gamma^t = \alpha^m = 0 = \gamma^s\beta = \beta\alpha^n, \; m = p^r, \; t = p^l, \; n \mid m, \; s\mid t. $$
\item[(ii)]
The algebra from (i) with the additional relation $ \gamma^{e}\beta = \beta\alpha^f$ where $e \mid t$ and $f \mid m$.
\end{itemize}
Again we consult the list of Bongartz-Gabriel and find that (up to duality) only the following values for $m,n,s$ and $t$ give a representation-finite algebra:

For the algebra from (i):
\begin{itemize}
\item any combination of $m\in \lbrace 1,2 \rbrace$, $n \in \lbrace 1,2 \rbrace$, $t = 2^l, l\geq 0$ and $s \in \lbrace 1,2 \rbrace$;
\item $m,s \geq 3$ and $n = 1 = t$.
\end{itemize}
For case (ii) the combination $ t = m = s = n = 3$ and $n = 1 = t$ is the only case with finite representation type that is not contained in the ones mentioned above.

Again this fits into our assertion and no other cases can occur, which finishes the proof.
\end{proof}
\begin{Rem}
For the representation-finite EI-categories with two simples one can compute the Auslander-Reiten quiver, since they are either hereditary of Dynkin type or occur in the list of Bongartz-Gabriel. For the latter case one again uses covering-theory to knit the Auslander-Reiten quiver of the covering and then pushes everything down to the algebra itself.
\end{Rem}

\bibliographystyle{plain}
\bibliography{refs2}

\begin{thebibliography}{10}

\bibitem{Alperin}
J.~L. Alperin.
\newblock {\em Local representation theory}, volume~11 of {\em Cambridge
  Studies in Advanced Mathematics}.
\newblock Cambridge University Press, Cambridge, 1986.
\newblock Modular representations as an introduction to the local
  representation theory of finite groups.

\bibitem{Bautistaetal}
R.~Bautista, P.~Gabriel, A.~V. Ro{\u\i}ter, and L.~Salmer{\'o}n.
\newblock Representation-finite algebras and multiplicative bases.
\newblock {\em Invent. Math.}, 81(2):217--285, 1985.

\bibitem{Bongartz-Gabriel}
K.~Bongartz and P.~Gabriel.
\newblock Covering spaces in representation-theory.
\newblock {\em Invent. Math.}, 65(3):331--378, 1981/82.

\bibitem{BLO}
Carles Broto, Ran Levi, and Bob Oliver.
\newblock The homotopy theory of fusion systems.
\newblock {\em J. Amer. Math. Soc.}, 16(4):779--856 (electronic), 2003.

\bibitem{Gabriel}
P.~Gabriel.
\newblock The universal cover of a representation-finite algebra.
\newblock In {\em Representations of algebras ({P}uebla, 1980)}, volume 903 of
  {\em Lecture Notes in Math.}, pages 68--105. Springer, Berlin, 1981.

\bibitem{GabrielZisman}
P.~Gabriel and M.~Zisman.
\newblock {\em Calculus of fractions and homotopy theory}.
\newblock Ergebnisse der Mathematik und ihrer Grenzgebiete, Band 35.
  Springer-Verlag New York, Inc., New York, 1967.

\bibitem{Lueck}
Wolfgang L{\"u}ck.
\newblock {\em Transformation groups and algebraic {$K$}-theory}, volume 1408
  of {\em Lecture Notes in Mathematics}.
\newblock Springer-Verlag, Berlin, 1989.
\newblock Mathematica Gottingensis.

\bibitem{Mitchell}
Barry Mitchell.
\newblock Rings with several objects.
\newblock {\em Advances in Math.}, 8:1--161, 1972.

\bibitem{Serre}
Jean-Pierre Serre.
\newblock {\em Linear representations of finite groups}.
\newblock Springer-Verlag, New York, 1977.
\newblock Translated from the second French edition by Leonard L. Scott,
  Graduate Texts in Mathematics, Vol. 42.

\bibitem{Fei1}
Fei Xu.
\newblock Representations of categories and their applications.
\newblock {\em J. Algebra}, 317(1):153--183, 2007.

\end{thebibliography}
\end{document}